\documentclass[12pt]{article}
\usepackage {graphics} 
\usepackage {graphicx}  
\usepackage {pstricks,pst-node,  pst-coil} 
\onecolumn  
\usepackage{amsmath,amssymb,latexsym}
\usepackage{amsmath,amsthm}
\newtheorem{tw}{Theorem}  
\newtheorem{lem}[tw]{Lemma}  
 
\newtheorem{obs}[tw]{Observation}

\newcommand{\pusty}{\emptyset}

\begin{document}
\hyphenation{every}
\title{On uniqueness of packing of  three copies of 2-factors}
\author{Igor Grzelec\thanks{Department of Discrete Mathematics, AGH University of Krakow, Poland. The corresponding author. Email:  grzelec@agh.edu.pl}, Tomáš Madaras\thanks{Institute of mathematics, P.J. Šafárik University, Košice, Slovakia}, Alfréd Onderko\footnotemark[2]}

\maketitle
\begin{abstract}
The packing of three copies of a graph $G$ is the union of three edge-disjoint copies (with the same vertex set) of $G$.
In this paper, we completely solve the problem of the uniqueness of packing of three copies of 2-regular graphs. In particular, we show that $C_3,C_4,C_5,C_6$ and $2C_3$ have no packing of three copies, $C_7,C_8,C_3 \cup C_4, C_4 \cup C_4, C_3 \cup C_5$ and $3C_3$ have unique packing, and any other collection of cycles has at least two distinct packings.
\end{abstract}

\section{Introduction}
 All graphs considered in this paper are finite, undirected and have neither loops nor multiple edges. For a graph $G$, we will denote its order $|V(G)|$ and size $|E(G)|$ as $n$ and $m$, respectively.

At the beginning, we present additional definitions which will be useful to formulate the results. For two graphs $G_1$ and $G_2$ with disjoint vertex sets, the {\it union} $G=G_1\cup G_2$  has $V(G)=V(G_1)\cup V(G_2)$ and  $E(G)=E(G_1)\cup E(G_2)$. The union of $n \geq 2$  disjoint copies of a graph $H$ is denoted by $G=nH$. Further, for graphs $G_1$ and $G_2$ such that $V(G_1)=V(G_2)$ and $E(G_1)\cap E(G_2)=\pusty$, the {\it edge sum} $G_1\oplus G_2$ has $V(G)=V(G_1)=V(G_2)$ and $E(G)=E(G_1)\cup E(G_2)$.

A permutation $\sigma$ on $V(G)$ with the property that whenever $xy \in E(G)$, then $\sigma (x)\sigma (y) \not \in E(G)$, is called an {\it embedding} of $G$ in its complement $\overline{G}$. In other words, an embedding is an edge-disjoint {\it packing} of two copies of $G$ into a complete graph $K_n$.

One of the first results on packing problem is the following theorem, which was proved independently in \cite{BoEl}, \cite{BuSc} and \cite{SaSp}:

\begin{tw} \label{E2}
Let $G=(V,E)$ be a graph of order $n$ and size $m$. If $m\leq n-2$ then $G$ can be embedded in its complement $\overline{G}$.  
\end{tw}

We can easily see that the star $K_{1,n-1}$ is not embeddable. Therefore,  Theorem \ref{E2} cannot be improved by raising the size of $G$. Burns and Schuster in \cite{BuSc1} described the full characterization of graphs of order $n$ and size $n-1$ that are embeddable:   

\begin{tw} \label{E3}
Let $G=(V,E)$ be a graph of order $n$ and size $m$. If $m\leq n-1$ then either $G$ is embeddable or $G$ is isomorphic to one of the following graphs: $K_{1, n-1}$,  $K_{1, n-4} \cup K_3$ with $n\geq8$, $K_1 \cup K_3$, $K_2 \cup K_3$, $K_1 \cup 2K_3$, $K_1 \cup C_4$.
\end{tw}

Considering the problem of the uniqueness of graph packings, let us explain first what we mean by distinct embeddings: let $\sigma$ be an embedding of the graph $G=(V,E)$. Denote by  $\sigma (G)$ the graph with the vertex set $V$ and the edge set   $\sigma ^* (E)$ where the mapping $\sigma ^*$ is induced by  $\sigma$. From the definition of an embedding we can see that the sets $E$ and $\sigma ^* (E)$ are disjoint; thus, we may create the graph $G\oplus\sigma (G)$. 
We say that two embeddings $\sigma_1$, $\sigma_2$ of a graph $G$ are {\it distinct} if the graphs $G\oplus\sigma_1 (G)$ and $G\oplus\sigma_2 (G)$ are not isomorphic. We call a graph $G$ {\it uniquely embeddable} if, for all embeddings $\sigma$ of $G$, all graphs $G\oplus\sigma (G)$ are isomorphic.

This problem has been the subject of three papers. The next theorem from \cite{Woz1} characterizes all graphs of order $n$ and size $n-2$ that are uniquely embeddable:

\begin{tw} \label{e1}
Let $G$ be a graph of order $n$ and size $m=n-2$. Then either $G$ is not uniquely embeddable or $G$ is isomorphic to one of the six following graphs: $K_2\cup K_1$, $2K_2$, $K_3\cup 2K_1$, $K_3\cup K_2\cup K_1$, $K_3\cup 2K_2$, $2K_3\cup 2K_1$.  
\end{tw}

The following characterization of uniquely embeddable forests was proved in \cite{Otwinowska}. 

\begin{tw} 
Let $F$ be a forest of order $n$ having at least one edge. Then either $F$ is not uniquely embeddable or $F$ is isomorphic to one of the following graphs: $K_2 \cup K_1$, $2K_2$, $3K_2$, the double star $S(p,q)$ or the $(n-1)$-vertex star with one edge subdivided. 
\end{tw}

Recently the problem of the uniqueness of embeddings was completely solved for 2-factors, i.e. a vertex-disjoint union of cycles. More precisely, the following theorem was proved in \cite{Grzelec Wozniak}:

\begin{tw}\label{Sums}
Let $G$ be a vertex-disjoint union of $k$ cycles. If $G$ is $C_3$, $C_4$ or  $2C_3$ then $G$ is not embeddable. The graphs $C_5$, $C_6$, $C_3 \cup C_4$, $C_3 \cup C_5$, $3C_3$ and $4C_3$ are uniquely embeddable. In every other case, there exist at least two distinct embeddings of $G$.
\end{tw}

For other results on different packing problems, we refer the reader to the survey papers \cite{Yap}, \cite{Woz5} and \cite{Woz6}.

To generalize the above mentioned problem of the uniqueness of packing for three copies of a graph, we first clarify  what we mean by  packing of three  copies of a graph, and when two such packings are distinct.
Let $G_1$, $G_2$ and $G_3$ be three copies of a graph $G$ of order $n$. We say that there exists a  packing of three  copies of $G$ into a complete graph $K_n$ if there exist injections $\alpha_i: V(G_i) \rightarrow V(K_n)$,  $i \in \{1, 2, 3\}$, such that, for $i\ne j$, $\alpha_i^*(E(G_i))\cap\alpha_j^*(E(G_j))=\emptyset$, where the mapping $\alpha_i^*: E(G_i) \rightarrow E(K_n)$ is induced by $\alpha_i$.
Two packings of three copies of $G$ are distinct if the graphs $\alpha_1(G)\oplus\alpha_2(G)\oplus\alpha_3(G)$ and $\alpha_1'(G)\oplus\alpha_2'(G)\oplus\alpha_3'(G)$ are not isomorphic. We say that a packing of three copies of $G$ is unique if all packings of three copies of $G$ are isomorphic. 

Note that the problem of the uniqueness of packing of three copies of a graph can be easily  generalized for $l\geq 4$ copies (the definitions are analogous).

It is worth to mention two results on the existence of packing of three copies of a graph. In \cite{Wozniak Wojda}, the following theorem (which yields full characterization of all graphs of order $n$ and size $n-2$ that have a packing of three copies) was proved: 

\begin{tw}
Let $G$ be a graph of order $n$ and size $m$. If $m \leq n-2$, then either there
exists a packing of three copies of $G$ or $G \in \{K_3 \cup 2K_1, K_4 \cup 4K_1\}$.
\end{tw}

In \cite{Wang Sauer}, the full characterization of all trees that have a packing of three copies was proved:

\begin{tw}
Let $T$ be a tree of order $n\geq6$ which is neither a star nor a star with one edge subdivided, nor else a 6-vertex star with one edge subdivided twice. Then
there exists a packing of three copies of $T$.
\end{tw}

The purpose of this paper is to consider the problem of the uniqueness of packing of three copies of a graph for 2-factors. This problem for 2-factors is related to the well-known Oberwolfach problem which is still open. To present this problem in a formal way, we provide additional definitions. By a \textit{$2$-factorization} of a graph $G$ we mean an edge-disjoint partition of the edge set of $G$ into 2-factors. A 1-factor of a graph $G$ (i.e. a perfect matching of $G$) will be denoted by $I$. The Oberwolfach problem (OP for short) asks whether a complete graph $K_n$ (for $n$ odd), or $K_n$ without a 1-factor (for $n$ even), admits a 2-factorization in which each 2-factor is isomorphic to a given 2-factor. More precisely, an instance OP$(n; n_1, \ldots, n_k)$ of the Oberwolfach problem asks if there is a 2-factorization of $K_n$ for $n$ odd, or $K_n \setminus I$ for $n$ even, such that each 2-factor is isomorphic to $C_{n_1}\cup C_{n_2}\cup \ldots \cup C_{n_k}$. Since the problem was posed in 1967 by Gerhard Ringel, many papers on the topic were published. With an exception of four cases, namely OP$(6; 3^2)$, OP$(9; 4, 5)$, OP$(11; 3^2, 5)$ and OP$(12; 3^4)$, for which solutions do not exist, solutions were obtained for all orders $n \leq 40$ (see \cite{OP1} and \cite{OP2}) and for many special cases (for example OP$(n; r^k, n-rk)$ for all $n \geq 6kr-1$, see \cite{OP3}). For more results on this topic, we refer the reader to the survey \cite{OP4}.

Now, we introduce the relation between the problem of the uniqueness of packing of three copies of a 2-factor and the Oberwolfach problem. Because the sum of three edge disjoint copies (with common vertex set) of any 2-factor is a 6-regular graph, 2-factors for which there exists a packing of three copies shall have order at least seven . 

\begin{obs}\label{existence} 
For any $2$-factor $G=C_{n_1}\cup C_{n_2}\cup \ldots \cup C_{n_k}$ of order $n=n_1+\ldots + n_k$, where $n\geq7$, if there exists a solution for the instance OP$(n; n_1, \ldots, n_k)$ of the Oberwolfach problem then there exists a packing of three copies of $G$.
\end{obs}
Remark that the converse of Observation \ref{existence} does not hold; it suffices to find a packing of three copies of $C_4\cup C_5$, because OP$(9; 4, 5)$ has no solution. The packing of three copies of this 2-factor is presented in Section \ref{Remaining_small_cases}. Moreover, considering complements of packings of three copies of 2-factor of order seven and eight, we can easily see that they are isomorphic to a graph of size zero and order seven, and to a perfect matching of order eight, respectively. Therefore the following observation holds:

\begin{obs}\label{uniqueness}
For any $2$-factor $G=C_{n_1}\cup C_{n_2}\cup \ldots \cup C_{n_k}$ of order $n=n_1+\ldots + n_k$, where $n\in\{7,8\}$, if there exists a solution for the instance OP$(n; n_1, \ldots, n_k)$ of the Oberwolfach problem, then there exist a unique packing of three copies of $G$.
\end{obs}

Note that the above observations can be generalized to the problem of the uniqueness of packing of $l\geq 4$ copies of a graph using similar reasoning. For any 2-factor $G=C_{n_1}\cup C_{n_2}\cup \ldots \cup C_{n_k}$ of order $n=n_1+\ldots + n_k$ with $n\geq 2l+1$ such that OP$(n; n_1, \ldots, n_k)$ has a solution, it is easy to see that there exists a packing of $l$ copies of $G$. Furthermore, for any 2-factor $G=C_{n_1}\cup C_{n_2}\cup \ldots \cup C_{n_k}$ of order $n=n_1+\ldots + n_k$, where $n\in\{2l+1, 2l+2\}$, if there exists a solution for  OP$(n; n_1, \ldots, n_k)$, then there exists a unique packing of $l$ copies of $G$. 

Now, we introduce our main result:

\begin{tw}\label{main}
Let $G=C_{n_1}\cup C_{n_2}\cup \ldots \cup C_{n_k}$ be a vertex-disjoint union of $k$ cycles. For cycles $C_3$, $C_4$, $C_5$, $C_6$ and the graph $2C_3$, there is no packing of three copies. The graphs $C_7$, $C_8$, $C_3 \cup C_4$, $C_4 \cup C_4$, $C_3 \cup C_5$ and $3C_3$ have unique packing of three copies. For any other graph $G$, there exist at least two distinct packings of its three copies.
\end{tw}

The proof of Theorem \ref{main} is presented in the next sections. Section 2 contains
the case of cycles, Section 3 presents the proof of the existence of packing of three copies of 2-factors (for $k\geq2$) and general strategy of the remaining part of the proof, Section 4 contains the proof of the existence of at least two distinct packings of three copies of 2-factors for five particular families of 2-factors, and the last section presents the proof for the remaining small cases.

\smallskip
\noindent
{\bf Remark.}  To better differentiate between copies of $G$ in a packing (both in subsequent figures and proofs), we say that the first (initial) copy of $G$ is \textit{black}, the second copy of $G$ is \textit{red} and the third copy of $G$ is \textit{blue} in such packing of $G$; this is useful, in particular, when a packing is presented solely by a figure.

For the proof of our main result, we will need the following lemma which generalizes Lemma~6 from \cite{Grzelec Wozniak}.

\begin{lem}
If a graph $G = C_{n_1}\cup C_{n_2}\cup \ldots \cup C_{n_k}$ has a packing $\alpha$ of three copies  such that the graph $\alpha_1(G)\oplus\alpha_2(G)\oplus\alpha_3(G)$ is not connected $($a disconnected packing$)$, then $G$ has another packing $\alpha'$ such that the graph $\alpha_1'(G)\oplus\alpha_2'(G)\oplus\alpha_3'(G)$ is connected $($a connected packing$)$. In particular, the graph $G$ has two distinct packings of three copies. 
\label{l1}
\end{lem} 
\begin{proof} 
Let $\alpha$ be the packing with the smallest number of connected components. If $H=\alpha_1(G)\oplus\alpha_2(G)\oplus\alpha_3(G)$ is connected, we are done.  Otherwise, let $H_1, H_2$ be two components of $H$.

Take a vertex $x_1$  of $H_1$ such that removing the two blue edges $x_1^-x_1$ and $x_1^+x_1$ (where $x_1^-$ and $x_1^+$ are neighbors of $x_1$ on the blue cycle in $H_1$) leaves
$H_1$ connected. In a similar manner, take $x_2$, a vertex belonging to the component $H_2$. Note that such selections of $x_1$ and $x_2$ are always possible --
it suffices to take, in $H_1$ and $H_2$, any vertex that is not a cut vertex (for example, the last vertex on the longest component path).

If now instead of the edges $x_1^-x_1$ and $x_1^+x_1$ we add two blue edges $x_2^-x_1$ and $x_2^+x_1$, and instead of the edges $x_2^-x_2$ and $x_2^+x_2 $ we add two blue edges $x_1^-x_2$ and $x_1^+x_2$, we obtain a new packing $\alpha'$ where two components $H_1$ and $H_2$ become one connected component. However, this is a contradiction with the choice of the packing $\alpha$.
\end{proof}

\section{Cycles}
In this section we prove the following lemma which will be useful in the remaining part of the proof:
\begin{lem}\label{k=1}
Let $C_n$ be a cycle of length $n$. For cycles $C_3$, $C_4$, $C_5$ and $C_6$, there is no packing of three copies.  The cycles $C_7$ and $C_8$ have unique packing of three copies. For longer cycles, there exist at least two distinct packings of three copies.
\end{lem}
\begin{proof}
Obviously, the cycles $C_3$, $C_4$, $C_5$ and $C_6$ do not have packing of three copies because such a packing is always a 6-regular graph. From Observation \ref{uniqueness} we know that $C_7$ and $C_8$ have unique packing of three copies.

We will denote by $C_n(a,b,c)$ the 6-regular circulant graph on $n$ vertices with generators $a$, $b$ and $c$, that is, the graph with vertex set $\mathbb{Z}_n=\{0, \ldots, n-1\}$ and edge set $\{\{x, x+s\}: x\in \mathbb{Z}_n, s\in \{a,b,c\} \}$  (note that the addition is modulo $n$). In \cite{Dean}, Dean proved that every 6-regular circulant graph on $n$ vertices with at least one generator of order $n$ (with respect to the group $\mathbb{Z}_n$) has Hamiltonian cycle decomposition. Thus, it suffices to find, for any $n\geq9$, two nonisomorphic 6-regular circulant graphs on $n$ vertices with at least one generator of order $n$; this confirms that there exist two distinct packings of three copies of a cycle with at least nine vertices.

To distinguish between 6-regular circulant graphs on $n\geq9$ vertices with at least one generator of order $n$, we use their chromatic number. The following two results from \cite{Meszka Nedela Rosa} and \cite{Heuberger} give information on the chromatic number of specific 6-regular circulant graphs:

\begin{tw}\label{c1}
Let $G = C_n(a, b, c)$ be a connected $6$-regular circulant graph, where
$n \geq 7$, $c = a + b$ or $n-c = a + b$ are pairwise distinct positive integers different from $n/2$. Let $\chi(G)$ be the chromatic number of $G$. Then
\begin{itemize}
\item $\chi(G)=7$ if and only if $G \cong K_7 \cong C_7(1, 2, 3)$,
\item $\chi(G)=6$ if and only if $G \cong C_{11}(1, 2, 3)$,
\item $\chi(G)=5$ if and only if $G \cong C_n(1, 2, 3)$ and $n \ne 7, 11$ is not \linebreak divisible by $4$, or $G$ is isomorphic to one of the following circulant graphs: $C_{13}(1, 3, 4)$, $C_{17}(1, 3, 4)$, $C_{18}(1, 3, 4)$, $C_{19}(1, 7, 8)$, $C_{25}(1, 3, 4)$, $C_{26}(1, 7, 8)$, $C_{33}(1, 6, 7)$, $C_{37}(1, 10, 11)$,
\item $\chi(G)=3$ if and only if $n$ is divisible by $3$ and none of $a$, $b$, $c$ is divisible by $3$,
\item $\chi(G)=4$ in all the remaining cases.
\end{itemize}
\end{tw}

\begin{tw}\label{c2}
Let $G$ be a connected circulant graph of order $n$. Then $G$ is bipartite if and only if $n$ is even and all generators are odd.
\end{tw}

From Theorem \ref{c1} we have that there exist two nonisomorphic  6-regular circulant graphs of order $n\geq9$ where $n \ne 11$ and $n\not\equiv 0 \pmod 4$: one can take, for example, $C_n(1, 2, 3)$ which has chromatic number equal to five and $C_n(1, 4, 5)$ which has chromatic number equal to three or four. The same theorem yields that $C_{11}(1, 4, 5)$ has chromatic number equal to four, and  $C_{11}(1, 2, 3)$ has chromatic number equal to six. Hence, it remains to find two nonisomorphic 6-regular circulant of graphs of order $n\geq9$ where $n$ is divisible by four. Using Theorem \ref{c2}, we know that $C_{4l}(1, 3, 5)$ is bipartite whereas $C_{4l}(1, 3, 4)$ is not bipartite. Therefore, there always exist two nonisomorphic 6-regular circulant graphs on $n\geq9$ vertices with one generator of order $n$.  
\end{proof}

\section{General strategy of the the proof for \\ 2-factors where $k\geq2$}
At first, we prove the following useful lemma about the existence of packing of three copies of 2-factors which contain at least two cycles:
\begin{lem}\label{eg}
Let $G$ be a vertex-disjoint union of $k\geq2$ cycles. Then $G$ has a packing of three copies except when $G=2C_3$.
\end{lem}
\begin{proof}
In the proof, we use the following Aigner and Brandt result from \cite{Aigner Brandt}:

\begin{tw}\label{AB}
Let $H$ be a graph of order $n$ with $\delta(H)\geq\frac{2n-1}{3}$. Then $H$ contains any graph $G$ of order at most $n$ with $\Delta(G)=2$ $($as a subgraph$)$.
\end{tw}

From Theorem \ref{Sums} we know that the graph $2C_3$ is the only one which is not embeddable, and therefore $2C_3$ also does not have a packing of three copies; hence, we can assume that $n>6$. From the fact that packing of two copies of a 2-factor is always a 4-regular graph, we obtain that its complement $H$ has $\delta(H)=n-1-4=n-5$. From Theorem \ref{AB} we can see that $H$ contains a packing of additional third copy of $G$ if $\delta(H)\geq\frac{2n-1}{3}$. Therefore, for every $n \geq 14$ we have $n-5 = \delta(H) \geq\frac{2n-1}{3}$, and so contains a packing of additional third copy of $G$. This proves that every 2-factor $G$ of order $n\geq 14$ has a packing of three copies.

Using Observation \ref{existence} and the Oberwolfach problem solutions for all orders $n\leq 40$ (see \cite{OP1} and \cite{OP2}), we get that every 2-factor $G$ of order $7\leq n\leq 14$ except for  $C_4\cup C_5$, $2C_3\cup C_5$ and $4C_3$ has a packing of three copies. The packings of three copies of $C_4\cup C_5$, $2C_3\cup C_5$ and $4C_3$ also exist, and they are presented in Section \ref{Remaining_small_cases}.
\end{proof}

Now, we present the general strategy of the remaining part of the proof of our main result. Assume that $G=C_{n_1}\cup C_{n_2}\cup \ldots \cup C_{n_k}$ is a vertex-disjoint union of $k\geq 2$ cycles, where $n=n_1+ \ldots +n_k$. Without loss of generality, assume that $n_1\leq n_2 \leq \ldots \leq n_k$. Note that for all such 2-factors, except for $2C_3$, a packing of three copies exists by Lemma \ref{eg}. From Observation \ref{uniqueness} and the Oberwolfach problem solutions for orders $n\in\{7,8\}$ we know that 2-factors on seven and eight vertices have unique packing of three copies. Therefore, we may assume that $n\geq9$. We consider several cases according to $k$. 

If $k=2$ then $G=C_{n_1}\cup C_{n_2}$ and $n_1\leq n_2$. If $n_1 \geq 7$, we have disconnected packing of $G$ which consists of two components. Each of these components we obtain as a packing of three copies of a cycle from Lemma \ref{k=1}. Thus, by Lemma \ref{l1}, the graph $G$ has two distinct packings of three copies. Therefore, we have to consider four families of 2-factors: $G = C_3 \cup C_x$ where $x \geq 6$, $G = C_4 \cup C_x$ where $x \geq 5$, $G = C_5 \cup C_x$ where $x \geq 5$ and $G = C_6 \cup C_x$ where $x \geq 6$. The uniqueness of packing of three copies of 2-factors from these families is investigated in Section \ref{construction} for $x\geq 11$ and in Section \ref{Remaining_small_cases} for $x\leq 10$.

If $k=3$ then $G=C_{n_1} \cup C_{n_2} \cup C_{n_3}$ and $n_1\leq n_2 \leq n_3$. We can divide $G$ into two subgraphs $G_1=C_{n_1} \cup C_{n_2}$ and $G_2=C_{n_3}$. Thus, from Lemmas \ref{k=1} and \ref{eg} we get a packing of three copies of $G_1$ and $G_2$ except for the case $G=C_3 \cup C_3 \cup C_x$ where $x \geq 3$, and the following twelve 2-factors: \\
$G=C_3 \cup C_4 \cup C_4$,  $G=C_3 \cup C_4 \cup C_5$, $G=C_3 \cup C_4 \cup C_6$, \\
$G=C_4 \cup C_4 \cup C_4$, $G=C_4 \cup C_4 \cup C_5$, $G=C_4 \cup C_4 \cup C_6$,\\
$G=C_4 \cup C_5 \cup C_5$, $G=C_4 \cup C_5 \cup C_6$,\\
$G=C_5 \cup C_5 \cup C_5$, $G=C_5 \cup C_5 \cup C_6$,\\
$G=C_5 \cup C_6 \cup C_6$,\\
$G=C_6 \cup C_6 \cup C_6$.\\
Therefore, by Lemma \ref{l1}, the graph $G$ has two distinct packings of three copies. The uniqueness of packing of three copies of 2-factors from the family $G=C_3 \cup C_3 \cup C_x$ where $x \geq 11$ is investigated in Section \ref{construction} and,  in Section \ref{Remaining_small_cases}, we investigate 2-factors from the family $G=C_3 \cup C_3 \cup C_x$ where $x \in \{3, 4, \ldots, 11\}$ and the above mentioned twelve exceptional 2-factors.

If $k=4$ then $G=C_{n_1} \cup C_{n_2} \cup C_{n_3} \cup C_{n_4}$ and $n_1\leq n_2 \leq n_3 \leq n_4$. If at least two $n_i$ (where $i \in \{1, 2, 3, 4\}$) are different from 3 then we can divide $G$ into two parts $G=G_1 \cup G_2$ such  that both $G_1$ and $G_2$ have packing of three copies by Lemma \ref{eg}. Therefore, by Lemma \ref{l1}, the graph $G$ has two distinct packings of three copies.
We argue similarly when $n_4\geq7$ (however, in this case, we need to use also Lemma \ref{k=1}). Thus, we are left with the following 2-factors: $G=C_3 \cup C_3 \cup C_3 \cup C_3$,  $G=C_3 \cup C_3 \cup C_3 \cup C_4$,  $G=C_3 \cup C_3 \cup C_3 \cup C_5$ and  $G=C_3 \cup C_3 \cup C_3 \cup C_6$; the uniqueness of packing of three copies of these four 2-factors is investigated in Section \ref{Remaining_small_cases}.

If $k=5$ then  $G=C_{n_1} \cup C_{n_2} \cup C_{n_3} \cup C_{n_4} \cup C_{n_5}$ and $n_1\leq n_2 \leq n_3 \leq n_4 \leq n_5$. If $n_5\geq4$, we can divide $G$ into two parts $G=G_1 \cup G_2$ such that $G_1=C_{n_1} \cup C_{n_2} \cup C_{n_3}$ and $G_2=C_{n_4} \cup C_{n_5}$ have packing of three copies by Lemma \ref{eg}. Therefore, by Lemma \ref{l1}, we know that $G$ has two distinct packings of three copies. Thus, we have to investigate the 
uniqueness of packing of three copies of $G=5C_3$; this will be done in  Section \ref{Remaining_small_cases}.

If $k\geq6$ then $G=C_{n_1} \cup C_{n_2}\cup \ldots \cup C_{n_k}$ and $n_1\leq n_2 \leq \ldots \leq n_k$. We can divide $G$ into two parts $G=G_1 \cup G_2$ so that $G_1=C_{n_1} \cup C_{n_2} \cup C_{n_3}$ and $G_2=C_{n_4} \cup C_{n_5} \cup \ldots \cup C_{n_k}$. From Lemma \ref{eg} and the fact that $k\geq 6$, we have packings of three copies of $G_1$ and $G_2$. Therefore, $G$ has a disconnected packing of three copies, and, from Lemma \ref{l1}, we get the connected one. 

\section{Five particular families of 2-factors}\label{construction}
In this section we present two distinct packings of three copies of 2-factors from five families: $C_3 \cup C_x$, $C_4 \cup C_x$, $C_5 \cup C_x$, $C_6 \cup C_x$ and $C_3 \cup C_3 \cup C_x$ where $x \geq 11$. We use the construction approach. The first pre\-sen\-ted packing of three copies of these 2-factors will contain a clique $K_5$ whereas the second one will not. At first, we present the packing with a clique $K_5$. 

We start our construction of packing of three copies of 2-factors with a clique $K_5$ from the smallest graphs in each family, that is, the graphs $C_3 \cup C_{11}$, $C_4 \cup C_{11}$, $C_5 \cup C_{11}$, $C_6 \cup C_{11}$ and $C_3 \cup C_3 \cup C_{11}$. The packing of three copies of these 2-factors is presented in Figures \ref{fig1} and \ref{fig2}.%
\begin{figure}[hp]
\centering
\includegraphics[width=1\textwidth]{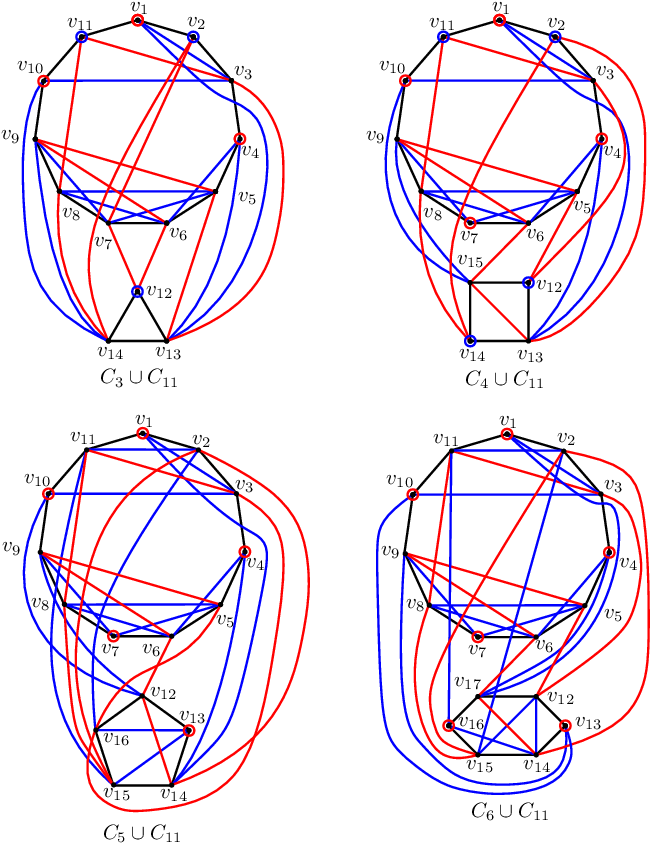}
\caption{The packing of three copies of $C_3 \cup C_{11}$, $C_4 \cup C_{11}$, $C_5 \cup C_{11}$ and $C_6 \cup C_{11}$ with a clique $K_5$.}
\label{fig1}
\end{figure}%
\begin{figure}[h]
\centering
\includegraphics[width=7cm]{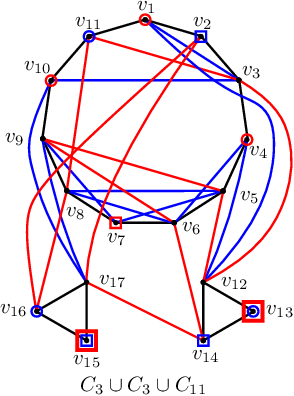}
\caption{The packing of three copies of $C_3 \cup C_3 \cup C_{11}$ with a clique $K_5$.}
\label{fig2}
\end{figure}
For the clarity of drawings, some cycles connecting vertices marked with the same color and type of the marker are not drawn. More precisely, in Figure \ref{fig1}, in the packing (of three copies) of $C_3 \cup C_{11}$, we have also the red cycle $v_1v_4v_{10}v_1$ and the blue cycle $v_2v_{12}v_{11}v_2$. In the packing of $C_4 \cup C_{11}$, we have also the red cycle $v_1v_4v_7v_{10}v_1$ and the blue cycle $v_2v_{12}v_{14}v_{11}v_2$. In the packing of $C_5 \cup C_{11}$, we have also the red cycle $v_1v_4v_{13}v_7v_{10}v_1$ and, in  the packing of $C_6 \cup C_{11}$, we have also the red cycle $v_1v_4v_{13}v_{16}v_7v_{10}v_1$. Similarly, in Figure \ref{fig2}, in the packing of three copies of $C_3 \cup C_3 \cup C_{11}$, we have also red cycles $v_1v_4v_{10}v_1$, $v_7v_{13}v_{15}v_7$ and blue cycles $v_2v_{14}v_{15}v_2$, $v_{11}v_{13}v_{16}v_{11}$. Note that each of these packings contains a subgraph $K_5$ induced by vertices $v_5$, $v_6$, $v_7$, $v_8$ and $v_9$. 

The presented packings of three copies of the smallest graphs from these particular families are easily extendable to appropriate graphs from these families if the longest cycle $C_x$ in the considered 2-factor has odd length. Note that in every already-presented packing, the vertices $v_{10}$, $v_{11}$, $v_1$, $v_2$ and $v_3$ induce almost the same subgraph (up to the edge $v_{11}v_2$ which is not present in the packing of $C_3\cup C_3 \cup C_{11}$). Therefore, we introduce a common extension for these packings in which we will change only the "upper part" of the packing, which contains vertices $v_{10}$, $v_{11}$, $v_1$, $v_2$, $v_3$ and edges incident to them. The method of extension depends on the number of edges added to the longest cycle in each copy of the smallest 2-factor in the respective family. To increase the length of the longest cycle in each copy of the smallest 2-factor in these families by 2, 4 or 6, we use appropriate extension presented in Figure \ref{fig3}.  
\begin{figure}[t]
\centering
\includegraphics[width=10.5cm]{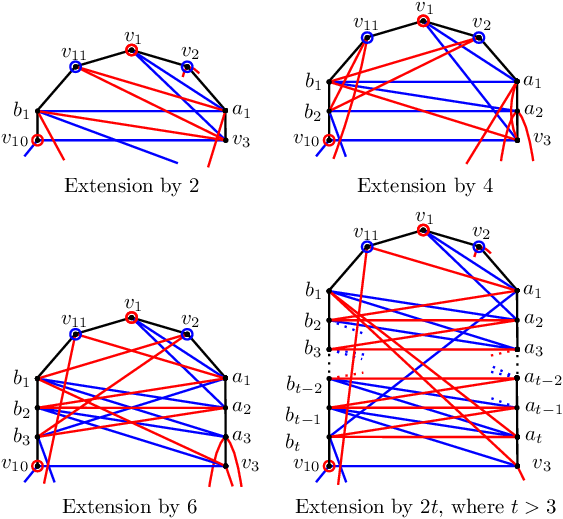}
\caption{The extensions of packings of three copies of 2-factors from five particular families when the longest cycle has odd length.}
\label{fig3}
\end{figure}
For better understanding, we describe each of these extensions in detail.

\textbf{Extension by 2.} At first, we replace black edges $v_2v_3$ and $v_{11}v_{10}$ by black paths $v_2a_1v_3$ and $v_{11}b_1v_{10}$, respectively. Then, we replace red edge $v_{11}v_8$ by red edge $b_1v_8$. Next, we replace the red edge from the "bottom part" to $v_{3}$ by the red edge from the same vertex in the "bottom part" to $a_1$. Then, we replace the blue edge from the "bottom part" to $v_{1}$ by the blue edge from the same vertex in the "bottom part" to $b_1$. At the end, we add red edges $b_1v_3$, $v_{11}a_1$ and the blue path $v_1a_1b_1$.

\textbf{Extension by 4.} At first, we replace black edges $v_2v_3$ and $v_{11}v_{10}$ by black paths $v_2a_1a_2v_3$ and $v_{11}b_1b_2v_{10}$, respectively. Then, we replace the red edge from the "bottom part" to $v_{3}$ by the red edge from the same vertex in the "bottom part" to $a_1$. Next, we replace the blue edge from the "bottom part" to $v_{1}$ by the blue edge from the same vertex in the "bottom part" to $b_2$. Then, we remove the red edge $v_{11}v_3$. Next, we replace two red edges from the "bottom part" to $v_2$ by two red edges from the same vertices in the "bottom part" to $a_2$. At the end, we add the red path $a_1v_3b_1v_2b_2v_{11}$ and the blue path $v_1a_1b_1a_2b_2$.

\textbf{Extension by 6.} At first, we replace black edges $v_2v_3$ and $v_{11}v_{10}$ by black paths $v_2a_1a_2a_3v_3$ and $v_{11}b_1b_2b_3v_{10}$, respectively. Then, we replace the blue edge from the "bottom part" to $v_{1}$ by the blue edge from the same vertex in the "bottom part" to $b_3$. Then, we remove the red edge $v_{11}v_3$ and the blue edge $v_1v_3$. Next, we replace two red edges from the "bottom part" to $v_2$ by two red edges from the same vertices in  the "bottom part" to $a_3$. At the end, we add the red path $v_3b_1v_2b_3a_2b_2a_1v_{11}$ and the blue path $v_3b_2a_3b_1a_2v_1a_1b_3$.

Now, we describe how to create the packing of three copies of a 2-factor with the longest cycle $C_{11+2t}$, for $t>3$, from the smallest 2-factor in the respective family. At first, we replace black edges $v_2v_3$ and $v_{11}v_{10}$ by black paths $v_2a_1a_2\ldots a_tv_3$ and $v_{11}b_1b_2\ldots b_tv_{10}$, respectively. 
Then, we replace the blue edge from the "bottom part" to $v_{1}$ by the blue edge from the same vertex in the "bottom part" to $b_t$. 
Then, we remove the red edge $v_{11}v_3$ and the blue edge $v_1v_3$. 
At the end, we add the red path $v_3b_1a_tb_ta_{t-1}b_{t-1}\ldots a_2b_2a_1v_{11}$ and the blue path $v_3b_{t-1}a_tb_{t-2}a_{t-1}\ldots b_2a_3b_1a_2v_1a_1b_t$. 
This extension of packing of three copies is also presented in Figure \ref{fig3}. 

Now, we show how to obtain a packing of three copies of a 2-factor from each of these families, if the longest cycle in 2-factor has an even length $x$. 
We take the packing of three copies of a 2-factor in which the longest cycle has length $x-1$. 
We replace: the black edge $v_1v_2$ by the black path $v_1v_cv_2$, the blue edge $v_{10}v_3$ by the blue path $v_{10}v_cv_3$, the red edge $v_{11}v_8$ by the red path $v_{11}v_cv_8$  (if $x=14$, we replace the red edge $v_{11}a_1$  by the red path $v_{11}v_ca_1$). Thus, we obtain a packing of three copies of a 2-factor from each of these families if the longest cycle in 2-factor has an even length. Note that in each packing which we obtain using the above extensions we have induced subgraph $K_5$ on vertices $v_5$, $v_6$, $v_7$, $v_8$ and $v_9$.

Next, we present packings of three copies of 2-factors from these five particular families without a clique $K_5$. 
The following observation is particularly useful for the construction:
\begin{obs}\label{obs: extension_K5_free}
Let $B$ be a union of cycles, let $q \geq 3$, and let $G$ be an instance of a $K_5$-free packing of three copies of $B \cup C_q$.
If there are three independent edges $e_1$, $e_2$, and $e_3$ on $C_q$ in the black, the blue, and the red copy of $B \cup C_q$, then there is a $K_5$-free packing of three copies of $B \cup C_{q+1}$.
\end{obs}
\begin{proof}
Let $e_1 = x_1x_2$, $e_2 = y_1y_2$, and $e_3 = z_1z_2$.
Let $w$ be a new vertex.
Replace edges $x_1x_2,y_1 y_2, z_1z_2$ with edges $x_1w, wx_2, y_1w, wy_2, z_1w, wz_2$.
The new graph $G'$ is a packing of three copies of $C_3 \cup C_{q+1}$.
Moreover, since no edges were added between the vertices of $V(G)$, the 5-clique, if there is one, contains $w$.
Let there be a 5-clique $A$ containing $w$ in $G'$.
This clique contains four neighbors of $w$, and, therefore, it contains two neighbors of $w$ which were adjacent in $G$ but are not adjacent in $G'$; this, however,  contradicts the fact that $A$ is a clique.
\end{proof}

Note that, instead of adding one new vertex, it is possible to add $k$ vertices at once, if there are $k$ pairwise edge-disjoint matchings, 
each consisting of three edges of $G$ lying on $C_q$ in the black, the blue, and the red subgraph of $G$, respectively.  
We will use this fact later.

\begin{figure}[h]
    \centering
    \includegraphics[width=6.5cm]{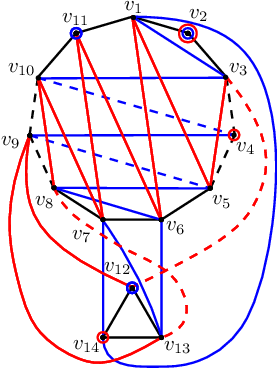}
    \caption{$K_5$-free packing of three copies of $C_3 \cup C_{11}$.}
    \label{fig: C3_C11_K5-free}
\end{figure}

Consider first the packing of three copies of $C_3 \cup C_{11}$ in Figure~\ref{fig: C3_C11_K5-free}; note that this graph is $K_5$-free, which can be checked using computer. We now show how to extend the packing in Figure~\ref{fig: C3_C11_K5-free} to a $K_5$-free packing of three copies of $C_3 \cup C_{11 + 4t}$ for some positive integer $t$.
Denote by $G$ the graph in Figure~\ref{fig: C3_C11_K5-free}.
Let $G'$ be a graph obtained from $G$ after removing edges $v_3v_4, v_3v_{12}, v_4v_5, v_4v_{10}, v_5v_9, v_8v_9, v_8v_{13}, v_9v_{10}$ (dashed edges in Figure~\ref{fig: C3_C11_K5-free}), adding new vertices $a_i, b_i, c_i, d_i$ for $i \in \{1, \dots, t\}$, and adding edges of the (black) paths $v_3(a_i)_{i=1}^tv_9$, $v_5(b_i)_{i=1}^tv_4$, $v_{10}(c_i)_{i=1}^tv_9$ and $v_8(d_i)_{i=1}^tv_9$,
(blue) paths $v_{10}(a_ic_i)_{i=1}^tv_4$ and $v_5(d_ib_i)_{i=1}^tv_9$,
and (red) paths $v_3(b_ia_i)_{i=1}^tv_{12}$ and $v_8(c_id_i)_{i=1}^tv_{13}$.
For an overview of the added part, see Figure~\ref{fig: C3_C11_K5-free_added_part_4t}. 

\begin{figure}[hp]
    \centering
    \includegraphics[width=12.5cm]{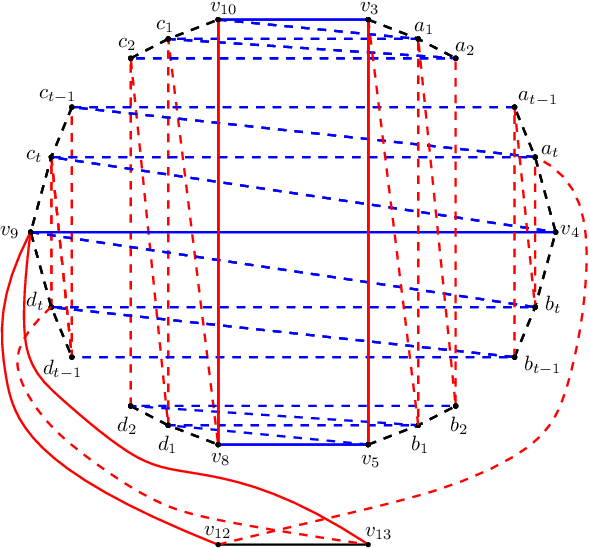}
    \caption{The subgraph $H$ of the $K_5$-free packing of three copies of $C_3 \cup C_{11+4t}$ induced on $4t$ added vertices and their neighbors. Dashed edges are newly added edges, full edges are old edges.}
    \label{fig: C3_C11_K5-free_added_part_4t}
\end{figure}

In the following, we refer to vertices from $\{v_1, \dots, v_{14}\}$ as old vertices, and other vertices of $G'$ as new vertices.

It follows from the construction that $G'$ is a packing of three copies of $C_3 \cup C_{11+4t}$, however, the absence of $K_5$ in $G'$ is not clear. 
Therefore, suppose to the contrary that there is a copy of $K_5$ in $G'$. Let $H$ be a subgraph of $G'$ induced on new vertices and their neighbors (see Figure~\ref{fig: C3_C11_K5-free_added_part_4t}).
Note that no edge between two old vertices was added in the construction; hence, if there is a clique on five vertices, it contains at least one new vertex.
Thus, each 5-clique in $G'$ is a 5-clique in $H$.
Note also that $H-a_1$ is 3-degenerate (consider for example the ordering $v_{12}, v_{13}, v_3, v_{10}, v_8, v_5, c_1, d_1, \dots, c_t, d_t, v_9, b_1, a_2, b_2, \dots a_t, b_t, v_4$);
hence, if there is a 5-clique in $H$, then it contains $a_1$.
However, $a_1$ and its neighbors induce a planar graph, see Figure~\ref{fig: neighbors_a1}.
Hence, there is no 5-clique in $G$.

\begin{figure}[hp]
    \centering
    \includegraphics[width=10cm]{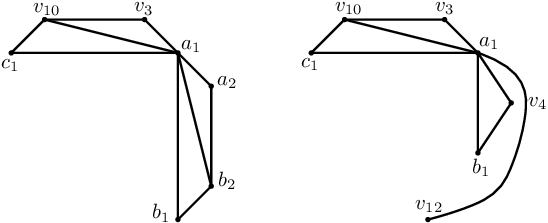}
    \caption{The subgraph induced on $N(a_1) \cup \{a_1\}$ for $t \geq 2$ (left) and $t = 1$ (right).}
    \label{fig: neighbors_a1}
\end{figure}

Using Observation~\ref{obs: extension_K5_free} for all three (pairwise edge-disjoint) matchings $v_1v_3, v_6v_{11}, v_7v_8$, $v_1v_5, v_6v_{13}, v_{10}v_{11}$, and $v_1v_{11}, v_3v_{10}, v_9v_{12}$ at once, we get that there is a $K_5$-free packing of $C_3 \cup C_{11 + 4t + q}$ for every nonnegative integer $t$ and every $q \in \{0,1,2,3\}$.
This covers all the cases of packings of three copies of $C_3 \cup C_x$, $x \geq 11$.
Moreover, such a construction does not remove any of the edges of the initial black, blue, or red copy of $C_3$; this will be useful to extend created $K_5$-free packings of three copies of $C_3 \cup C_x$ to $K_5$-free packings of $C_y \cup C_x$ for $y \in \{4,5,6 \}$. 
To obtain a packing of three copies of $C_4 \cup C_x$, apply Observation~\ref{obs: extension_K5_free} for $B = C_x$, $q = 4$, $e_1 = v_{12} v_{13}$, $e_2 = v_2v_{11}$, and $e_3 = v_4v_{14}$; for later use, denote the newly added vertex by $w_1$.

Note that there are no larger sets of matchings of desired properties in the packing of three copies of $C_3 \cup C_x$ that could be used to extend them to a packing of three copies of $C_5 \cup C_x$ or $C_6 \cup C_x$ at once.
However, we can do it in steps. 
Observe that, in the created packing of three copies of $C_4 \cup C_x$, the black edge $w_1v_{13}$, the blue edge $v_{11}v_{12}$, and the red edge $v_2 v_{14}$ form a matching, and we can make use of Observation~\ref{obs: extension_K5_free} to obtain a $K_5$-free packing of three copies of $C_5 \cup C_x$; denote the newly added vertex by $w_2$ (see Figure~\ref{fig: extending_triangles}).
To obtain a packing of three copies of $C_6 \cup C_x$, simply repeat the previous step for edges $v_{13}v_{14}, v_{11}w_2, v_2v_4$.

\begin{figure}[h]
    \centering
    \includegraphics[width=13.7cm]{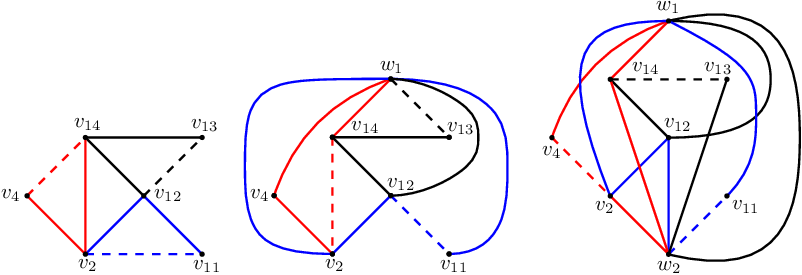}
    \caption{Extensions of monochromatic triangles. Dashed edges are edges of a matching for which Observation~\ref{obs: extension_K5_free} is used.}
    \label{fig: extending_triangles}
\end{figure}

Finally we show how to obtain a $K_5$-free packing of three copies of \linebreak$C_3 \cup C_3 \cup C_x$ for $x \geq 11$. The process is similar to the case of $K_5$-free packings of three copies of $C_3 \cup C_x$. We start with the initial packing $G$ of three copies of $C_3 \cup C_3 \cup C_{11}$ displayed in Figure~\ref{fig: C3_C3_C11_K5_free}. Using a computer check, one can find that such a packing is $K_5$-free (in this case, the size of the maximum clique is three).

\begin{figure}[h]
    \centering
    \includegraphics[width=7cm]{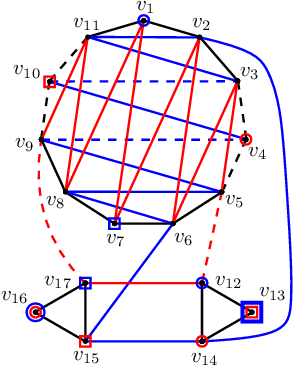}
    \caption{$K_5$-free packing of three copies of $C_3 \cup C_3 \cup C_{11}$.}
    \label{fig: C3_C3_C11_K5_free}
\end{figure}

We show how to extend $G$ to a $K_5$-free packing of three copies of $C_3 \cup C_3 \cup C_{11+4t}$ for some positive integer $t$, using a similar approach as previously.
Let $G'$ be a graph obtained from $G$ by adding $4t$ new vertices, namely $a_i,b_i,c_i,d_i$ for $i \in \{1, \dots, t\}$, removing the edges  $v_3v_4, v_3v_{10}, v_4v_5, v_4v_9,v_5v_{12}, v_9v_{10}, v_9v_{17}, v_{10}v_{11}$ (dashed edges in Figure~\ref{fig: C3_C3_C11_K5_free}), and adding 
(black) paths $v_{11}(a_i)_{i=1}^tv_{10}$, $v_3(b_i)_{i=1}^tv_4$, $v_9(c_i)_{i=1}^tv_{10}$, and $v_5(d_i)_{i=1}^tv_4$, 
(blue) paths $v_{11}(b_ia_i)_{i=1}^tv_4$ and $v_5(c_id_i)_{i=1}^tv_{10}$, 
and (red) paths $v_9(a_ic_i)_{i=1}^tv_{17}$ and $v_5(b_id_i)_{i=1}^tv_{12}$ (see Figure~\ref{fig: C3_C3_C11_K5-free_added_part_4t}).

\begin{figure}[h]
    \centering
    \includegraphics[width=12.5cm]{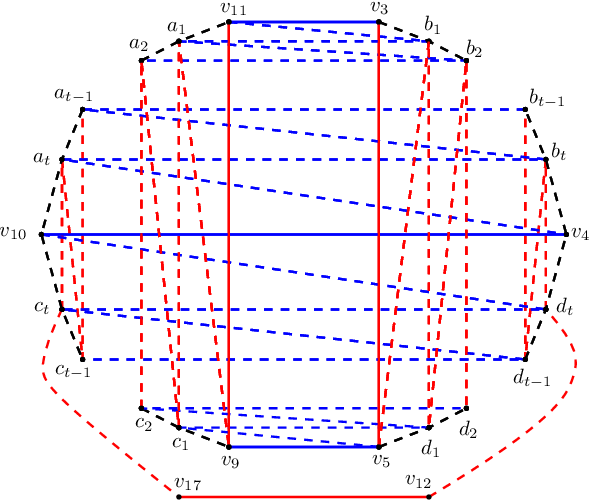}
    \caption{The subgraph $H$ induced on new vertices and their neighborhood in the $K_5$-free packing of three copies of $C_3 \cup C_3 \cup C_{11+4t}$. Dashed edges are newly added edges.}
    \label{fig: C3_C3_C11_K5-free_added_part_4t}
\end{figure}

Clearly, the presented construction creates a packing of three copies of $C_3 \cup C_3 \cup C_{11+4t}$. Suppose that a 5-clique $A$ was created in the process. Since no edge between two old vertices (that is, the vertices from $\{v_1, \dots, v_{17}\}$) was added, at least one of the vertices of $A$ is a new vertex. Hence, the 5-clique is present in a subgraph $H$ induced on new vertices and their neighbors, see Figure~\ref{fig: C3_C3_C11_K5-free_added_part_4t}.

Vertices $v_{12}$ and $v_{17}$ are of degree two in $H$, and the vertex $v_3$ is of degree three in $H$, hence, none of them is in $A$.

Vertices $v_4$, $v_9$, $v_{10}$ and $v_{11}$ are of degree four in $H$. Thus, if any of them is contained in the 5-clique $A$, then all its neighbors are from $A$; however, for each of these vertices, there are two of its neighbors that are not adjacent. Namely, for $v_9$, the vertices $v_5$ and $v_{11}$ are not adjacent, for $v_{11}$, the vertices $v_3$ and $v_9$ are not adjacent, and, for $v_4$ and $v_{10}$, the vertices $a_t$ and $d_t$ are not adjacent. Since the degree of $v_5$ in $H$ is five and at least two of its neighbors are not from the 5-clique $A$ (namely $v_3$ and $v_9$), we get that $v_5 \notin A$. Hence, $A$ contains only new vertices. However, the subgraph of $G'$ induced on new vertices is 3-degenerate (consider the vertex ordering $b_1, a_1, b_2 \dots, a_t, c_1, d_1, c_2, \dots, d_t$) and, therefore, it does not contain a 5-clique.

Now, we show how to extend the constructed $K_5$-free packing of three copies of $C_3 \cup C_3 \cup C_{11+4t}$ to a $K_5$-free packing of three copies of \linebreak $C_3 \cup C_3 \cup C_{11+4t+q}$ for $q  \in \{1,2,3\}$.
In all cases, use Observation~\ref{obs: extension_K5_free} for three (pairwise edge-disjoint) matchings $v_1v_2, v_{14}v_{15}, v_3v_5$, $v_2v_3, v_6 v_{15}, v_9v_{11}$, and $v_6 v_7, v_5v_8, v_9v_{17}$.
This completes the proof in case of five particular families of 2-factors.

\section{Remaining small cases}\label{Remaining_small_cases}

This section contains the discussion on the uniqueness of packings of three copies of small 2-factors which were not treated by general constructions in Section \ref{construction}. First, note that the packing of three copies of $3C_3$ is unique, as it corresponds to Steiner triple system STS(9) on nine points (equivalently, to the affine plane of order three), which is unique. For the remaining 46 small 2-factors, Table \ref{46_small_cases} contains description of two distinct packings of three copies of particular 2-factors; each of three 2-factors is presented as a collection of sequences of vertices of its cycles. Moreover, we present two distinct packings of three copies of these 2-factors for which the Oberwolfach problem has not solution in Figures \ref{fig4}, \ref{fig5} and \ref{fig6}.

\begin{figure}[hp]
\centering
\includegraphics[width=\textwidth]{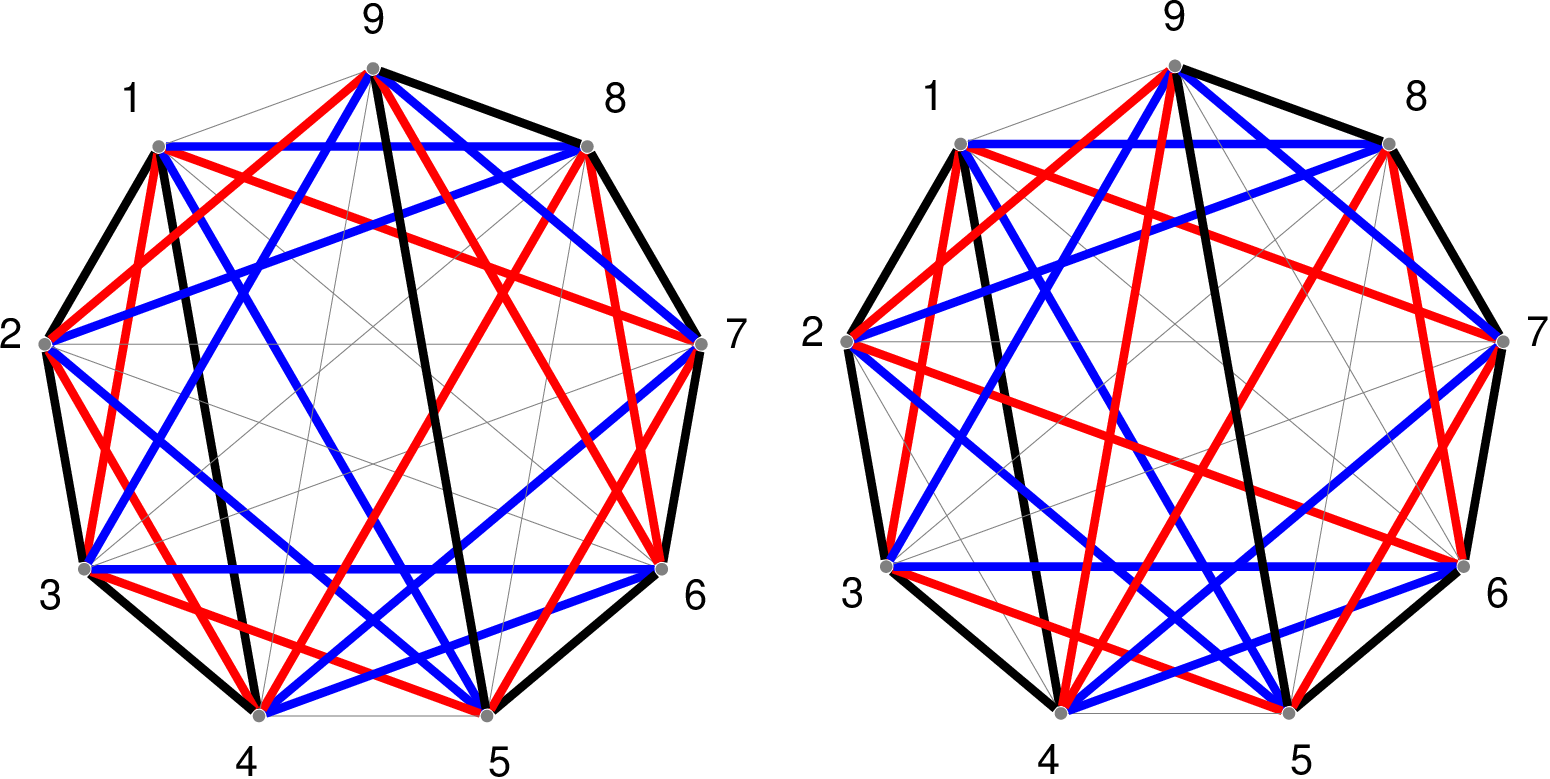}
\caption{Two distinct packings of three copies of $C_4 \cup C_5$.}
\label{fig4}
\end{figure}

\begin{figure}[hp]
\centering
\includegraphics[width=\textwidth]{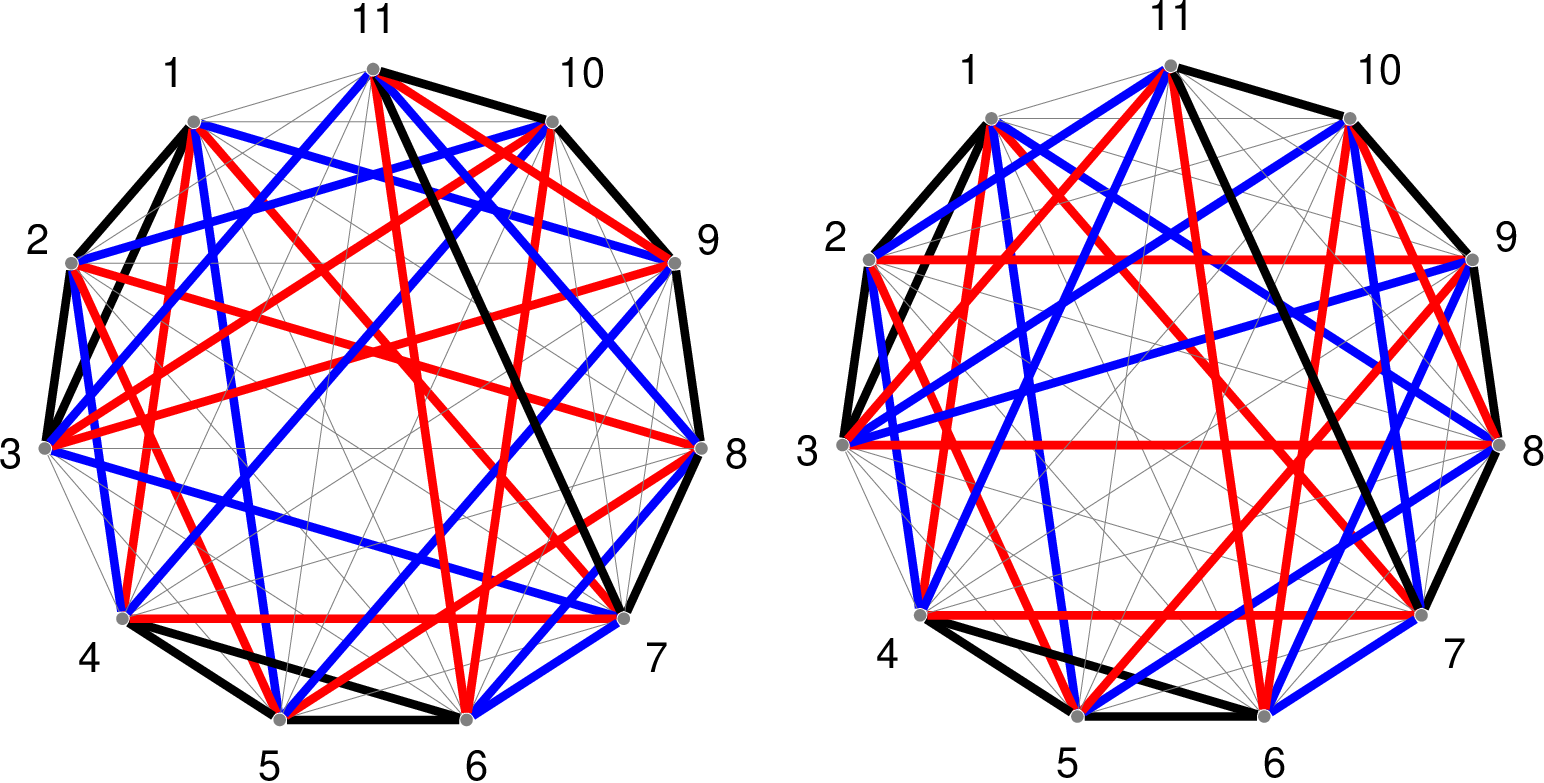}
\caption{Two distinct packings of three copies of $2C_3 \cup C_5$.}
\label{fig5}
\end{figure}

\begin{figure}[h]
\centering
\includegraphics[width=\textwidth]{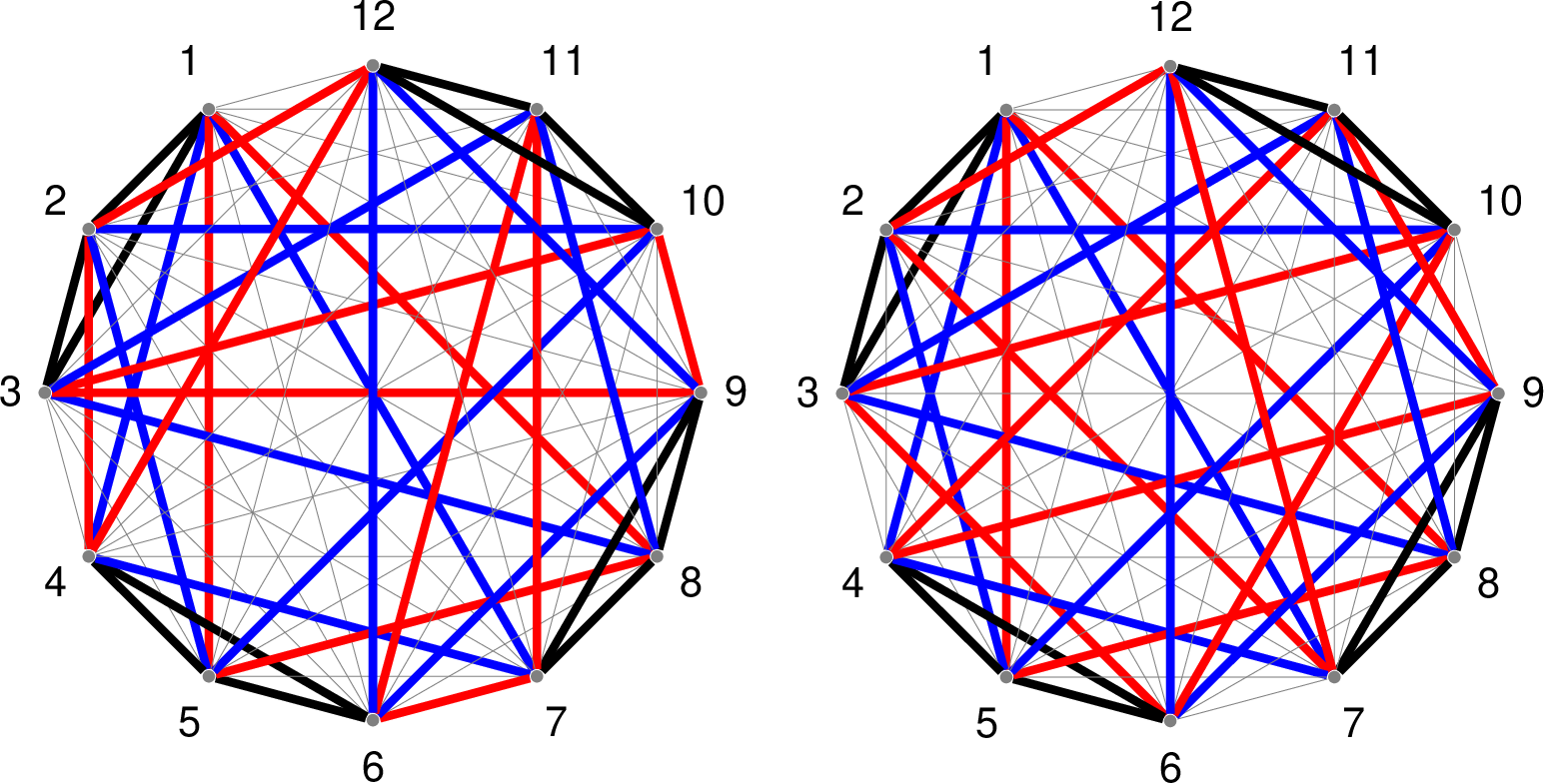}
\caption{Two distinct packings of three copies of $4C_3$.}
\label{fig6}
\end{figure}

For computer-assisted finding of these distinct packings, we used Wolfram Mathematica computer algebra system with its graph theory procedures. About the half of cases was solved using the following strategy: considering a 2-factor $H \cong C_{n_1} \cup \ldots \cup C_{n_k}$, we first removed from the complete graph $K_{n_1 + \ldots + n_k}$ the cycles $(1, \dots, n_1), (n_1 + 1, n_1 + 2, \ldots, n_1 + n_2), \dots , (n_1 + \ldots + n_{k-1}+1, \dots, n_1 + n_2 + \ldots + n_k)$. In the resulting graph $G_1$, we were looking for several distinct subgraphs isomorphic to $H$ (using the procedure {\tt FindIsomorphicSubgraph[$G$,$H$,$p$]} which allows to find either all, or at most $p$  distinct copies of $H$ in $G$). Among the graphs resulted from removing these subgraphs from $G_1$, we were looking for two nonisomorphic graphs $G_2',G_2''$, and, in them, we again looked for distinct subgraphs isomorphic to $H$. Finally, in two collections of graphs obtained from $G_2',G_2''$ by third removal of isomorphic copies of $H$, we were able to find two nonisomorphic graphs; their complements yielded the desired distinct packings.

We have to notice that this strategy failed when $k = 3, n_1 \geq 4, n_2 \geq 5$ and, also, when $k \geq 4$: the procedure {\tt FindIsomorphicSubgraph[$G_1$,$H$,$p$]} was able to find only at most three distinct copies of $H$ in $G_1$ (higher values of $p$ resulted in computation crash), and further attempts to look for copies of $H$ in $G_2', G_2''$ have led to computation crash or to isomorphic graphs. To overcome these obstacles, we have generated, for each of the remaining cases, a collection of distinct packings of two copies of $H$ using the code in Python (which took a fixed 2-factor, then renamed its vertices using a random permutation to obtain another 2-factor with the same cycle structure, and then checked whether these two 2-factors are edge disjoint and forming together a 4-regular graph). These packings of two copies of $H$ were first removed from $K_{n_1 + \ldots + n_k}$ and, in each of the obtained graphs, a single third copy of $H$ was searched for (again using {\tt FindIsomorphicSubgraph[]} procedure). Among the graphs resulted from this removal, we searched for a pair of nonisomorphic ones (again, their complements yielded the desired packings).

\begin{table}[hp]
\begin{tiny}
\begin{center}
\begin{tabular}{||c|l|l||}
\hline
\hline
2-factor & First packing & Second packing \\
\hline\hline
$C_3 \cup C_6$ & $(1,2,3),(4,5,6,7,8,9) $ & $(1,2,3),(4,5,6,7,8,9) $  \\
 & $(3,4,6), (1,8,2,5,7,9)$ & $(3,4,6), (1,5,9,7,2,8)$ \\
 & $(2,6,9), (1,4,8,5,3,7)$ & $(2,6,9),(1,4,8,5,3,7) $ \\
$C_3 \cup C_7$ & $(1,2,3),(4,5,6,7,8,9,10)$ & $(1,2,3),(4,5,6,7,8,9,10)$ \\
& $(1,4,6),(2,5,3,8,10,7,9)$ & $(1,4,6),(2,5,7,9,3,8,10)$\\
& $(1,5,7),(2,4,8,6,9,3,10)$ & $(1,5,8), (2,4,3,7,10,6,9)$\\
$C_3 \cup C_8$ & $(1,2,3),(4,5,6,7,8,9,10,11)$ & $(1,2,3),(4,5,6,7,8,9,10,11)$ \\
&$(1,4,6),(2,5,3,7,9,11,8,10)$ & $(1,4,6),(2,5,3,8,10,7,11,9)$\\
&$(1,5,7),(2,8,3,9,4,10,6,11)$ & $(1,5,7),(2,4,8,11,3,9,6,10)$\\
$C_3 \cup C_9$ & $(1,2,3),(4,5,6,7,8,9,10,11,12)$ & $(1,2,3),(4,5,6,7,8,9,10,11,12)$ \\
& $(1,4,6),(2,5,3,7,9,11,8,10,12)$ & $(1,4,6),(2,5,3,7,9,11,8,12,10)$\\
& $(1,5,7),(2,4,8,12,9,3,10,6,11)$ & $(1,5,7),(2,4,8,10,3,9,12,6,11)$\\
$C_3 \cup C_{10}$ & $(1,2,3),(4,5,6,7,8,9,10,11,12,13)$ & $(1,2,3),(4,5,6,7,8,9,10,11,12,13)$\\
& $(1,4,6),(2,5,3,7,9,11,8,12,10,13)$ & $(1,4,6),(2,5,3,7,9,11,8,13,10,12)$\\
& $(1,5,7),(2,4,3,8,10,6,11,13,9,12)$ & $(1,5,7),(2,4,3,6,10,8,12,9,13,11)$\\
$C_4 \cup C_5$ & $(1,2,3,4),(5,6,7,8,9)$ & $(1,2,3,4),(5,6,7,8,9)$\\
& $(1,3,5,7),(2,4,8,6,9)$ & $(1,3,5,7),(2,6,8,4,9)$\\
& $(1,5,2,8),(3,6,4,7,9)$ & $(1,5,2,8),(3,6,4,7,9)$\\
$C_4 \cup C_6$ & $(1,2,3,4),(5,6,7,8,9,10)$ & $(1,2,3,4),(5,6,7,8,9,10)$ \\
& $(1,3,5,7),(2,4,8,10,6,9)$ & $(1,3,5,7),(2,6,8,10,4,9)$\\
& $(1,5,2,8),(3,6,4,9,7,10)$ & $(1,5,2,8),(3,9,6,4,7,10)$\\
$C_4 \cup C_7$ & $(1,2,3,4),(5,6,7,8,9,10,11)$ & $(1,2,3,4),(5,6,7,8,9,10,11)$\\
& $(1,3,5,7),(2,4,6,9,11,8,10)$ & $(1,3,5,7),(2,4,8,10,6,9,11)$\\
& $(1,5,2,8),(3,6,10,4,9,7,11)$ & $(1,5,2,6),(3,8,11,4,9,7,10)$\\
$C_4 \cup C_8$ & $(1,2,3,4),(5,6,7,8,9,10,11,12)$ & $(1,2,3,4),(5,6,7,8,9,10,11,12)$\\
& $(1,3,5,7),(2,4,6,8,10,12,9,11)$ & $(1,3,5,7),(2,4,6,9,11,8,10,12)$\\
& $(1,5,2,6),(3,9,4,10,7,11,8,12)$ & $(1,5,2,6),(3,8,12,9,4,10,7,11)$\\
$C_4 \cup C_9$ & $(1,2,3,4),(5,6,7,8,9,10,11,12,13)$ & $(1,2,3,4),(5,6,7,8,9,10,11,12,13)$\\
& $(1,3,5,7),(2,4,6,8,10,12,9,11,13)$ & $(1,3,5,7),(2,4,6,8,10,12,9,13,11)$\\
& $(1,5,2,6),(3,7,9,13,10,4,11,8,12)$ & $(1,5,2,6),(3,7,9,11,4,10,13,8,12)$\\
$C_4 \cup C_{10}$ & $(1,2,3,4),(5,6,7,8,9,10,11,12,13,14)$ & $(1,2,3,4),(5,6,7,8,9,10,11,12,13,14)$\\
& $(1,3,5,7),(2,4,6,8,10,12,9,13,11,14)$ & $(1,3,5,7),(2,4,6,8,10,12,9,14,11,13)$\\
& $(1,5,2,6),(3,7,4,9,11,8,12,14,10,13)$ & $(1,5,2,6),(3,7,4,8,11,9,13,10,14,12)$\\
$C_5 \cup C_5$ & $(1,2,3,4,5),(6,7,8,9,10)$ & $(1,2,3,4,5),(6,7,8,9,10)$\\
& $(1,3,5,6,8),(2,4,9,7,10)$ & $(1,3,5,2,6),(4,8,10,7,9)$\\
& $(1,4,6,2,7),(3,8,10,5,9)$ & $(1,4,7,2,10),(3,6,8,5,9)$\\
$C_5 \cup C_6$ & $(1,2,3,4,5),(6,7,8,9,10,11)$ & $(1,2,3,4,5),(6,7,8,9,10,11)$\\
& $(1,3,5,2,4),(6,8,10,7,11,9)$ & $(1,3,5,2,6),(4,7,9,11,8,10)$\\
& $(1,6,2,7,9),(3,8,4,10,5,11)$ & $(1,4,2,7,11),(3,8,5,9,6,10)$\\
$C_5 \cup C_7$ & $(1,2,3,4,5),(6,7,8,9,10,11,12)$ & $(1,2,3,4,5),(6,7,8,9,10,11,12)$\\
& $(1,3,5,2,4),(6,8,10,7,12,9,11)$ & $(1,3,5,2,4),(6,8,11,7,9,12,10)$\\
& $(1,6,2,7,9),(3,8,4,10,12,5,11)$ & $(1,6,2,7,10),(3,8,4,9,11,5,12)$\\
$C_5 \cup C_8$ & $(1,2,3,4,5),(6,7,8,9,10,11,12,13)$ & $(1,2,3,4,5),(6,7,8,9,10,11,12,13)$\\
& $(1,3,5,2,4),(6,8,10,7,11,13,9,12)$ & $(1,3,5,2,4),(6,8,10,7,12,9,13,11)$\\
& $(1,6,2,7,9),(3,8,4,11,5,12,10,13)$ & $(1,6,2,7,9),(3,8,4,11,5,12,10,13)$\\
$C_5 \cup C_9$ & $(1,2,3,4,5),(6,7,8,9,10,11,12,13,14)$ & $(1,2,3,4,5),(6,7,8,9,10,11,12,13,14)$\\
& $(1,3,5,2,4),(6,8,10,7,9,12,14,11,13)$ & $(1,3,5,2,4),(6,8,10,7,11,13,9,14,12)$\\
& $(1,6,2,7,11),(3,8,4,9,13,5,12,10,14)$ & $(1,6,2,7,9),(3,8,4,11,14,5,12,10,13)$\\
$C_5 \cup C_{10}$ & $(1,2,3,4,5),(6,7,8,9,10,11,12,13,14,15)$ & $(1,2,3,4,5),(6,7,8,9,10,11,12,13,14,15)$\\
& $(1,3,5,2,4),(6,8,10,7,9,11,13,15,12,14)$ & $(1,3,5,2,4),(6,8,10,7,9,12,14,11,15,13)$\\
& $(1,6,2,7,11),(3,8,4,12,5,13,9,14,10,15)$ & $(1,6,2,7,11),(3,8,4,9,13,5,12,15,10,14)$\\
$C_6 \cup C_6$ & $(1,2,3,4,5,6),(7,8,9,10,11,12)$ & $(1,2,3,4,5,6),(7,8,9,10,11,12)$\\
& $(1,3,5,2,4,7),(6,8,10,12,9,11)$ & $(1,3,5,2,4,7),(6,9,11,8,12,10)$\\
& $(1,4,6,2,7,9),(3,10,5,11,8,12)$ & $(1,4,6,2,7,11),(3,8,10,5,9,12)$\\
$C_6 \cup C_7$ & $(1,2,3,4,5,6),(7,8,9,10,11,12,13)$ & $(1,2,3,4,5,6),(7,8,9,10,11,12,13)$ \\
& $(1,3,5,2,4,7),(6,8,10,12,9,11,13)$ & $(1,3,5,2,4,7),(6,8,10,12,9,13,11)$\\
& $(1,4,6,2,7,9),(3,10,13,5,11,8,12)$ & $(1,4,6,2,7,9),(3,10,13,5,11,8,12)$\\
$C_6 \cup C_8$ & $(1,2,3,4,5,6),(7,8,9,10,11,12,13,14)$ &$(1,2,3,4,5,6),(7,8,9,10,11,12,13,14)$ \\
& $(1,3,5,2,4,7),(6,8,10,12,9,13,11,14)$ & $(1,3,5,2,4,7),(6,8,10,12,9,14,11,13)$\\
& $(1,4,6,2,7,5),(3,9,11,8,12,14,10,13)$ & $(1,4,6,2,7,5),(3,8,11,9,13,10,14,12)$\\
$C_6 \cup C_9$ & $(1,2,3,4,5,6),(7,8,9,10,11,12,13,14,15)$ & $(1,2,3,4,5,6),(7,8,9,10,11,12,13,14,15)$ \\
& $(1,3,5,2,4,7),(6,8,10,12,9,13,15,11,14)$ & $(1,3,5,2,4,7),(6,8,10,12,9,14,11,13,15)$\\
& $(1,4,6,2,7,5),(3,8,11,9,14,12,15,10,13)$ & $(1,4,6,2,7,5),(3,8,11,9,13,10,14,12,15)$\\
$C_6 \cup C_{10}$ & $(1,2,3,4,5,6),(7,8,9,10,11,12,13,14,15,16)$ & $(1,2,3,4,5,6),(7,8,9,10,11,12,13,14,15,16)$\\
& $(1,3,5,2,4,7),(6,8,10,12,9,11,14,16,13,15)$ & $(1,3,5,2,4,7),(6,8,10,12,9,13,15,11,14,16)$\\
& $(1,4,6,2,7,5),(3,8,11,13,9,14,10,15,12,16)$ & $(1,4,6,2,7,5),(3,8,11,9,14,12,15,10,13,16)$\\
\hline
$2C_3\cup C_4$ & $(1,2,3),(4,5,6),(7,8,9,10)$ & $(1,2,3),(4,5,6),(7,8,9,10)$ \\
& $(1,4,7),(2,5,8),(3,9,6,10)$ & $(1,4,7),(2,5,9),(3,8,6,10)$\\
& $(1,5,9),(2,4,10),(3,7,6,8)$ & $(1,5,8),(2,4,10),(3,7,6,9)$\\
$2C_3\cup C_5$ & $(1,2,3),(4,5,6),(7,8,9,10,11)$ & $(1,2,3),(4,5,6),(7,8,9,10,11)$\\
& $(1,4,7),(2,5,8),(3,9,11,6,10)$ & $(1,4,7),(2,5,9),(3,8,10,6,11)$\\
& $(1,5,9),(2,4,10),(3,7,6,8,11)$ & $(1,5,8),(2,4,11),(3,9,6,7,10)$\\
\hline\hline
\end{tabular}
\end{center}
\end{tiny}
\end{table}

\begin{table}[hp]
\begin{tiny}
\makebox[\textwidth][c]{
\begin{tabular}{||c|l|l||}
\hline
\hline
2-factor & First packing & Second packing \\
\hline\hline
$2C_3\cup C_6$ & $(1,2,3),(4,5,6),(7,8,9,10,11,12)$ & $(1,2,3),(4,5,6),(7,8,9,10,11,12)$\\
& $(1,4,7),(2,5,8),(3,6,10,12,9,11)$ & $(1,4,7),(2,5,8),(3,9,11,6,10,12)$\\
& $(1,5,9),(2,4,10),(3,7,6,11,8,12)$ & $(1,5,9),(2,4,10),(3,6,12,8,11,7)$\\
$2C_3\cup C_7$ & $(1,2,3),(4,5,6),(7,8,9,10,11,12,13)$ & $(1,2,3),(4,5,6),(7,8,9,10,11,12,13)$ \\
& $(1,4,7),(2,5,8),(3,6,9,11,13,10,12)$ & $(1,4,7),(2,5,8),(3,6,10,12,9,13,11)$\\
& $(1,5,9),(2,4,10),(3,7,11,6,12,8,13)$ & $(1,5,9),(2,4,10),(3,7,11,6,12,8,13)$\\
$2C_3\cup C_8$ & $(1,2,3),(4,5,6),(7,8,9,10,11,12,13,14)$ & $(1,2,3),(4,5,6),(7,8,9,10,11,12,13,14)$\\
& $(1,4,7),(2,5,8),(3,6,9,11,13,10,12,14)$ & $(1,4,7),(2,5,8),(3,6,9,11,13,10,14,12)$\\
& $(1,5,9),(2,4,10),(3,7,11,14,6,12,8,13)$ & $(1,5,9),(2,4,10),(3,7,11,14,6,12,8,13)$\\
$2C_3\cup C_9$ & $(1,2,3),(4,5,6),(7,8,9,10,11,12,13,14,15)$ & $(1,2,3),(4,5,6),(7,8,9,10,11,12,13,14,15)$ \\
& $(1,4,7),(2,5,8),(3,6,9,11,13,10,14,12,15)$ & $(1,4,7),(2,5,8),(3,6,9,11,13,10,15,12,14)$\\
& $(1,5,9),(2,4,10),(3,7,6,12,8,13,15,11,14)$ & $(1,5,9),(2,4,10),(3,7,6,12,8,14,11,15,13)$\\
$2C_3\cup C_{10}$ & $(1,2,3),(4,5,6),(7,8,9,10,11,12,13,14,15,16)$ & $(1,2,3),(4,5,6),(7,8,9,10,11,12,13,14,15,16)$\\
& $(1,4,7),(2,5,8),(3,6,9,11,13,10,14,16,12,15)$ & $(1,4,7),(2,5,8),(3,6,9,11,13,10,15,12,14,16)$\\
& $(1,5,9),(2,4,10),(3,7,6,8,12,14,11,15,13,16)$ & $(1,5,9),(2,4,10),(3,7,6,11,14,8,12,16,13,15)$\\
$C_3\cup 2C_4$ & $(1,2,3),(4,5,6,7),(8,9,10,11)$ & $(1,2,3),(4,5,6,7),(8,9,10,11)$\\
& $(1,4,6),(2,5,8,10),(3,7,9,11)$ & $(1,4,6),(2,5,8,10),(3,9,7,11)$\\
& $(1,5,7),(2,8,3,9),(4,10,6,11)$ & $(1,5,7),(2,4,9,11),(3,8,6,10)$\\
$C_3\cup C_4\cup C_5$ & $(1,2,3),(4,5,6,7),(8,9,10,11,12)$ & $(1,2,3),(4,5,6,7),(8,9,10,11,12)$\\
& $(1,4,6),(2,5,3,7),(8,10,12,9,11)$ & $(1,4,6),(2,5,3,8),(7,10,12,9,11)$\\
& $(1,5,8),(2,9,4,10),(3,6,11,7,12)$ & $(1,5,7),(2,9,3,10),(4,8,11,6,12)$\\
$C_3\cup C_4\cup C_6$ & $(1,2,3),(4,5,6,7),(8,9,10,11,12,13)$ & $(1,2,3),(4,5,6,7),(8,9,10,11,12,13)$\\
& $(1,5,7),(2,4,8,12),(3,9,11,6,10,13)$ & $(1,5,7),(2,8,3,11),(4,9,13,6,10,12)$\\
& $(1,4,6),(2,5,3,7),(8,10,12,9,13,11)$ & $(1,4,6),(2,5,3,7),(8,10,13,11,9,12)$\\
$3C_4$ & $(1,2,3,4),(5,6,7,8),(9,10,11,12)$ & $(1,2,3,4),(5,6,7,8),(9,10,11,12)$\\
& $(1,3,5,7),(2,4,9,11),(6,8,10,12)$ & $(1,3,5,7),(2,4,9,11),(6,10,8,12)$\\
& $(1,5,2,6),(3,9,7,10),(4,11,8,12)$ & $(1,5,2,6),(3,9,8,11),(4,7,10,12)$\\
$2C_4\cup C_5$ & $(1,2,3,4),(5,6,7,8),(9,10,11,12,13)$ & $(1,2,3,4),(5,6,7,8),(9,10,11,12,13)$\\
& $(1,3,5,7),(2,4,6,8),(9,11,13,10,12)$ & $(1,3,5,7),(2,4,6,8),(9,11,13,10,12)$\\
& $(1,5,2,9),(3,10,6,11),(4,7,12,8,13)$ & $(1,5,2,9),(3,10,7,11),(4,8,12,6,13)$\\
$2C_4\cup C_6$ & $(1,2,3,4),(5,6,7,8),(9,10,11,12,13,14)$ & $(1,2,3,4),(5,6,7,8),(9,10,11,12,13,14)$ \\
& $(1,6,9,12),(4,11,8,14),(2,5,3,13,10,7)$ & $(2,12,3,13),(4,6,10,7),(1,5,14,11,9,8)$\\
& $(1,3,6,8),(2,4,5,9),(7,12,10,14,11,13)$ & $(1,3,5,7),(2,4,8,10),(6,11,13,9,12,14)$\\
$C_4\cup 2C_5$ & $(1,2,3,4),(5,6,7,8,9),(10,11,12,13,14)$ & $(1,2,3,4),(5,6,7,8,9),(10,11,12,13,14)$\\
& $(2,9,4,14),(1,3,12,8,5),(6,10,7,13,11)$ & $(1,3,9,11),(2,5,14,7,12),(4,6,10,8,13)$\\
& $(1,6,2,7),(3,5,4,8,11),(9,13,10,12,14)$ & $(1,5,3,6),(2,4,7,9,14),(8,11,13,10,12)$\\
$C_4\cup C_5\cup C_6$ & $(1,2,3,4),(5,6,7,8,9),(10,11,12,13,14,15)$ & $(1,2,3,4),(5,6,7,8,9),(10,11,12,13,14,15)$\\
& $(3,14,8,15),(1,10,2,7,13),(4,5,11,6,9,12)$ & $(2,6,12,7),(1,10,14,5,13),(3,9,15,8,4,11)$\\
& $(1,3,5,7),(2,4,6,8,11),(9,10,13,15,12,14)$ & $(1,3,5,7),(2,4,6,8,10),(9,11,13,15,12,14)$\\
$3C_5$ & $(1,2,3,4,5),(6,7,8,9,10),(11,12,13,14,15)$ & $(1,2,3,4,5),(6,7,8,9,10),(11,12,13,14,15)$\\
& $(1,7,11,2,14),(3,5,9,6,15),(4,10,12,8,13)$ & $(1,7,5,2,15),(3,6,13,4,8),(9,12,10,11,14)$\\
& $(1,3,6,2,4),(5,7,9,11,8),(10,13,15,12,14)$ & $(1,3,5,6,4),(2,7,9,11,8),(10,13,15,12,14)$\\
$2C_5\cup C_6$ & $(1,2,3,4,5),(6,7,8,9,10),(11,12,13,14,15,16)$ & $(1,2,3,4,5),(6,7,8,9,10),(11,12,13,14,15,16)$\\
& $(1,4,16,10,12),(2,7,9,5,8),(3,14,6,13,11,15)$ & $(1,10,15,6,12),(2,7,4,14,16),(3,5,13,9,11,8)$\\
& $(1,3,5,2,6),(4,7,10,8,11),(9,12,14,16,13,15)$ & $(1,3,6,2,4),(5,7,9,12,8),(10,14,11,15,13,16)$\\
$C_5\cup 2C_6$ & $(1,2,3,4,5),(6,7,8,9,10,11),(12,13,14,15,16,17)$ & $(1,2,3,4,5),(6,7,8,9,10,11),(12,13,14,15,16,17)$\\
& $(4,10,14,16,13),(1,8,6,15,9,17),(2,7,12,3,5,11)$ & $(1,10,2,6,15),(3,7,16,14,11,9),(4,8,13,17,5,12)$\\
& $(1,3,6,2,4),(5,7,9,11,8,13),(10,15,17,14,12,16)$ & $(1,4,2,5,3),(6,8,10,7,9,13),(11,15,17,14,12,16)$\\
$3C_6$ & $(1,2,3,4,5,6),(7,8,9,10,11,12),(13,14,15,16,17,18)$ &$(1,2,3,4,5,6),(7,8,9,10,11,12),(13,14,15,16,17,18)$ \\
& $(1,5,14,7,3,13),(2,12,18,10,6,16),(4,8,11,17,9,15)$ & $(1,4,2,17,10,12),(3,5,8,18,16,13),(6,9,11,15,7,14)$\\
& $(1,3,5,2,4,7),(6,8,10,12,9,11),(13,15,17,14,18,16)$ & $(1,3,6,2,5,7),(4,8,10,13,9,12),(11,16,14,17,15,18)$\\

\hline
$4C_3$ & $(1,2,3),(4,5,6),(7,8,9),(10,11,12)$ & $(1,2,3),(4,5,6),(7,8,9),(10,11,12)$\\
& $(1,5,8),(2,4,12),(3,9,10),(6,7,11)$ & $(1,5,8),(2,7,12),(3,6,10),(4,9,11)$\\
& $(1,4,7),(2,5,10),(3,8,11),(6,9,12)$ & $(1,4,7),(2,5,10),(3,8,11),(6,9,12)$\\
$3C_3 \cup C_4$ & $(1,2,3),(4,5,6),(7,8,9),(10,11,12,13)$ & $(1,2,3),(4,5,6),(7,8,9),(10,11,12,13)$\\
& $(1,5,9),(2,4,10),(3,6,11),(7,12,8,13)$ & $(1,5,9),(2,4,10),(3,11,13),(6,7,12,8)$\\
& $(1,4,7),(2,5,8),(3,10,12),(6,9,11,13)$ & $(1,4,7),(2,5,8),(3,10,12),(6,11,9,13)$\\
$3C_3 \cup C_5$ & $(1,2,3),(4,5,6),(7,8,9),(10,11,12,13,14)$ & $(1,2,3),(4,5,6),(7,8,9),(10,11,12,13,14)$ \\
& $(2,9,10),(3,6,12),(4,7,14),(1,5,13,8,11)$ & $(1,4,12),(2,7,10),(5,9,11),(3,8,14,6,13)$\\
& $(1,4,8),(2,5,7),(3,9,11),(6,13,10,12,14)$ & $(1,5,7),(2,4,8),(3,6,9),(10,12,14,11,13)$\\
$3C_3 \cup C_6$ & $(1,2,3),(4,5,6),(7,8,9),(10,11,12,13,14,15)$ & $(1,2,3),(4,5,6),(7,8,9),(10,11,12,13,14,15)$\\
& $(1,9,11),(3,5,7),(8,12,15),(2,4,14,6,10,13)$ & $(2,4,8),(6,13,15),(10,12,14),(1,9,3,7,5,11)$\\
& $(1,4,7),(2,5,8),(3,6,12),(9,10,14,11,13,15)$ & $(1,4,7),(2,5,9),(3,6,14),(8,10,13,11,15,12)$\\
\hline
$5C_3$ & $(1,2,3),(4,5,6),(7,8,9),(10,11,12),(13,14,15)$ & $(1,2,3),(4,5,6),(7,8,9),(10,11,12),(13,14,15)$\\
& $(1,4,7),(2,10,13),(3,6,14),(5,8,11),(9,12,15)$ & $(1,4,7),(2,5,8),(3,12,15),(6,10,13),(9,11,14)$\\
& $(1,5,9),(2,4,8),(3,10,15),(6,11,13),(7,12,14)$ & $(1,5,9),(2,4,10),(3,6,14),(7,11,15),(8,12,13)$\\
\hline\hline
\end{tabular}
}
\end{tiny}
\caption{Two distinct packings of three copies of each of 46 small 2-factors}
\label{46_small_cases}
\end{table}


\begin{thebibliography}{99}
\bibitem{OP1}P. Adams, D. Bryant, \textit{Two-factorisations of complete graphs of
orders fifteen and seventeen}, Australas. J. Comb., {\bf (1)35}
(2006), 113--118.

\bibitem{Aigner Brandt}M. Aigner, S. Brandt,  \textit{Embedding arbitrary graphs of maximum degree two}, J. London Math. Soc., {\bf(2)48} (1993), 39--51.

\bibitem{BoEl}B. Bollob\'as, S. E. Eldridge, \textit{Packings of graphs and applications to computational complexity}, J. Combin. Theory Ser. B, {\bf 25} (1978), 105--124.

\bibitem{BuSc}D. Burns, S. Schuster, \textit{Every $(p,p-2)$ graph is contained in its complement}, J. Graph Theory, {\bf 1} (1977){\bf ,} 277--279.

\bibitem{BuSc1}D. Burns, S. Schuster, \textit{Embedding $(n,n-1)$ graphs   in   their complements}, Israel J. Math., {\bf 30} (1978), 313--320.

\bibitem{OP4}C. J. Colbourn, J. H. Dinitz, eds., \textit{The CRC Handbook of Combinatorial Designs, Second Edition}, CRC Press, Boca Raton, 2006.

\bibitem{Dean}M. Dean, \textit{On Hamilton Cycle Decomposition of 6-regular Circulant Graphs}, Graphs Combin., {\bf 22} (2006), 331--340.

\bibitem{OP2}A. Deza, F. Franek, W. Hua, M. Meszka, A. Rosa, \textit{Solutions to the Oberwolfach problem for orders $18$ to $40$}, J. Comb. Math. Comb. Comput.,  {\bf 74} (2010), 95--102.

\bibitem{Grzelec Wozniak}I. Grzelec, M. Pilśniak, M. Woźniak, \textit{A note on uniquely embeddable $2$-factors}, 
Appl. Math. Comput., {\bf 468} (2024), 128505.

\bibitem{Heuberger}C. Heuberger, \textit{On planarity and colorability of circulant
graphs}, Discrete Math., {\bf 268} (2003), 153--169.

\bibitem{OP3}A. J. W. Hilton, M. Johnson, \textit{Some results on the Oberwolfach
problem}, J. London Math. Soc., {\bf(2)64} (2001), 513--522.

\bibitem{Meszka Nedela Rosa}M. Meszka, R. Nedela, A. Rosa, \textit{Circulants and the chromatic index of Steiner triple systems}, Math. Slovaca, {\bf 56} (2006) 371--378.

\bibitem{Otwinowska}J. Otfinowska, M. Woźniak, \textit{A Note on Uniquely Embeddable Forests}, Discuss. Math. Graph Th.,  {\bf 33.1} (2013), 193--201.

\bibitem{SaSp}N. Sauer, J. Spencer, \textit{Edge disjoint placement of graphs}, J. Combin. Theory  Ser. B, {\bf 25} (1978){\bf }, 295--302.

\bibitem{Wang Sauer}H. Wang, N. Sauer, \textit{Packing three copies of a tree into a complete graph}, European J. Combin., {\bf 14} (1993), 137--142.

\bibitem{Woz1}M. Wo\'zniak,  \textit{A note on uniquely embeddable graphs}, Discuss. Math. Graph Th., {\bf 18} (1998), 15--21.

\bibitem{Woz5}M. Wo\'zniak, \textit{Packing of graphs}, Diss. Math., {\bf 362} (1997), pp.78.

\bibitem{Woz6}M. Wo\'zniak, \textit{Packing of graphs and permutation -- a survery}, Discrete Math., {\bf 276} (2004), 379--391.

\bibitem{Wozniak Wojda}M. Wozniak, A. P. Wojda, \textit{Triple placement of graphs}, Graphs Combin., {\bf 9} (1993), 85--91.

\bibitem{Yap}H. P. Yap, \textit{Packing of graphs -- a survey},
Discrete Math., {\bf 72} (1988), 395--404.
\end{thebibliography}
\end{document}